\documentclass[11pt]{amsart}
\usepackage{geometry}                
\usepackage{graphicx,stmaryrd}
\usepackage{amssymb,amsmath}
\usepackage{epstopdf}
\usepackage{tikz}
\usepackage{lscape}
\usepackage{color}
\usepackage{todonotes}
\usepackage{mathrsfs}
\usepackage{longtable}

\usepackage{hyperref}

\allowdisplaybreaks

\theoremstyle{plain}
\newtheorem{theorem}{Theorem}[section]

\newtheorem{corollary}[theorem]{Corollary}
\newtheorem{proposition}[theorem]{Proposition}
\newtheorem{lemma}[theorem]{Lemma}

\theoremstyle{definition}
\newtheorem{definition}[theorem]{Definition}
\newtheorem{example}[theorem]{Example}
\newtheorem{remark}[theorem]{Remark}

\newcommand{\I}{\mathcal{I}}
\newcommand{\T}{\mathcal{T}}

\newcommand{\A}{\mathcal{A}}
\newcommand{\B}{\mathcal{B}}
\newcommand{\M}{\mathcal{M}}

\newcommand{\NN}{\mathbb{N}}
\newcommand{\ZZ}{\mathbb{Z}}

\newcommand{\C}{\mathcal{C}}

\newcommand{\Z}{\mathbb{Z}}
\newcommand{\F}{\mathbb{F}}

\newcommand{\G}{\mathcal{G}}
\newcommand{\R}{\ensuremath \mathbb{R}}

\newcommand\restr[2]{{
  \left.\kern-\nulldelimiterspace 
  #1 
  \vphantom{|} 
  \right|_{#2} 
  }}

\newcommand{\op}{\mathrm{op}}

\renewcommand{\max}{\mathrm{max}}

\newcommand{\rk}{\mathrm{rk}}
\newcommand{\crk}{\mathrm{crk}}
\newcommand{\calP}{\mathcal{P}}

\newcommand{\bs}{\backslash}
\newcommand{\Fl}{\mathcal{F}l}

\newcommand{\gd}{\ensuremath{\delta}}

\newcommand{\gm}{\ensuremath{\mu}}

\newcommand{\gs}{\ensuremath{\sigma}}

\title{Chain Tutte polynomials}
\author{Max Wakefield}
\thanks{}
\address{US Naval Academy \footnote{The views expressed in this article are those of the author and do not reflect the official policy or position of the U.S. Naval Academy, Department of the Navy, the Department of Defense, or the U.S. Government.}\\
  572-C Holloway Rd\\
  Annapolis MD, 21402 USA}
\email[]{wakefiel@usna.edu}

\begin{document}
\maketitle

\begin{abstract} The Tutte polynomial and Derksen's $\mathcal{G}$-invariant are the universal deletion-contraction and valuative matroid and polymatroid invariants, respectively. There are only a handful of well known invariants (like the matroid Kazhdan-Lusztig polynomials) between (in terms of fineness) the Tutte polynomial and Derksen's $\mathcal{G}$-invariant. The aim of this study is to define a spectrum of generalized Tutte polynomials to fill the gap between the Tutte polynomial and Derksen's $\mathcal{G}$-invariant. These polynomials are built by taking repeated convolution products of universal Tutte characters studied by Dupont, Fink, and Moci and using the framework of Ardila and Sanchez for studying valuative invariants. We develop foundational aspects of these polynomials by showing they are valuative on generalized permutahedra and present a generalized deletion-contraction formula. We apply these results on chain Tutte polynomials to obtain formulas for the M\"obius polynomial, the opposite characteristic polynomial, a generalized M\"obius polynomial, Ford's expected codimension of a matroid variety, and Derksen's $\mathcal{G}$-invariant. \end{abstract}

\section{Introduction}

For any set $X$ we denote the set of subsets of $X$ by $2^X$. Also, we use $\NN=\{0,1,2,\dots\}$ and for $n\in \NN \backslash \{0\}$ we let $[n]=\{1,2,\dots ,n\}$. Notationally we will denote a sequence of objects $(n_1,\ldots,n_k)$ by $(n_i)_1^k$ or $(n_i)$ when the indices range is clear from the context. A \emph{polymatroid} is a pair $M=(\A,\rk)$ where $\A$ is a finite set, called the ground set and $\rk :2^\A\to \NN$ is a function that satisfies \begin{enumerate}

\item  $\rk (\emptyset)=0$

\item for all $X\subseteq Y\subseteq \A$, $\rk(X)\leq \rk(Y)$

\item for all $X\subseteq Y\subseteq \A$, $\rk(X)+\rk(Y) \geq \rk(X\cap Y) +\rk (X\cup Y)$.

\end{enumerate} A \emph{matroid} is a polymatroid that also satisfies for all $a\in \A$, $\rk(a)\in \{0,1\}$. There are many equivalent definitions of matroids and we follow the standard conventions in \cite{Oxley} and \cite{Ardila-15}. The set of \emph{independent sets} of a matroid is the set \[\I =\{X\subseteq \A | \ \rk(X)=|X|\} \] and one can define a matroid in terms of independent sets as in \cite[Sec 7.2]{Ardila-15}. Given a matroid defined by independent sets $\I$ one can define the rank function $\rk : 2^{\A}\to \NN$ by $\rk(X)=\max \{|I|\ |\ I\in \I, \ I\subseteq X \}.$ This gives a cryptomorphism between the rank function and the independent set definition (again see \cite[Sec 7.2]{Ardila-15}). We will utilize both of these definitions: when we write $M=(\A, \rk)$ we are considering the rank function and when we write $M=(\A ,\I )$ we are considering $\I$ as the set of independent sets. For any polymatroidal function, when the context of the matroid is not clear we will include the polymatroid as a subscript, for example we may denote $\rk$ as $\rk_M$. An element $a\in \A$ is a \emph{loop} in $M$ if $\rk(a)=0$ and a \emph{coloop} if $\rk(\A -a)=\rk(M)-1$. We also say that two elements $a,b\in \A$ which are not loops are \emph{parallel} if $\rk (\{a,b\})=1$. Finally we say that a matroid is \emph{simple} if it has no loops or parallel elements. 

The \emph{Whitney rank generating function of} a polymatroid $M=(\A,\rk)$ is 

$$W_M(a_1;b_1)=\sum\limits_{S\subseteq \A}a_1^{\rk(M)-\rk(S)}b_1^{|S|-\rk(S)}$$ and the \emph{Tutte polynomial} of $M$ is $$T_M(x;y)=W_M(x-1,y-1)=\sum\limits_{S\subseteq \A}(x-1)^{\rk(M)-\rk(S)}(y-1)^{|S|-\rk(S)}.$$

Tutte polynomials were originally defined on graphs by Tutte in \cite{Tutte-poly} while their coefficients were first investigated by Whitney in \cite{Whitney-32} (for some history see \cite{Farr07}). Crapo popularized and generalized Tutte polynomials to matroids in \cite{Crapo-69} . Since then Tutte polynomials have become arguably the most studied invariants on graphs, matroids, hyperplane arrangements, polymatroids, and extended generalized permutahedra. Recently an entire CRC handbook \cite{Tutte-handbook} edited by Ellis-Monaghan and Moffatt is entirely dedicated to the study of Tutte polynomials. Most importantly Tutte polynomials are universal among deletion-contraction invariants, meaning one can formulate any deletion-contraction invariant, like the characteristic polynomial of a matroid which agrees with the chromatic polynomial for graphs, as an evaluation of the Tutte polynomial. 

Recently there has been significant developments on understanding the foundational aspects of Tutte polynomials. Speyer in \cite{Speyer08} showed that the Tutte polynomial over matroids is valuative from the polytope perspective. The work of Dupont, Fink, and Moci in \cite{DFM-19} presents the Tutte polynomial as a specialization of Tutte characters on set species which allows them to examine Tutte polynomials on objects like Delta matroids and prove various convolution formulas. Ardila and Sanchez in \cite{AS-20} conducted a study of valuative invariants and built a framework for understanding them with convolutions. This led to another proof that the Tutte polynomial is valuative but over a much larger class of objects, in particular extended generalized permutahedra. Then, Ferroni and Schr\"oter in \cite{FS-22} studied in detail some foundational valuations on certain small matroid subdivisions. Further recent developments on the Tutte polynomial have brought new results on the structure of the coefficients \cite{CF-22}, different interpretations of the Tutte polynomial \cite{Kockol-21}, and the dimension of the space spanned by its coefficients \cite{Kung-17}.

In \cite{D08} Derksen defined an invariant for polymatroids which has since been studied by many authors. Now we present a definition of this invariant. A \emph{complete chain} of a polymatroid $M=(\A,\rk)$ is a sequence of subsets $\bar{X}=(X_i)_{0}^{|\A|}$ of $\A$ such that $X_0=\emptyset \subset X_1 \subset \cdots \subset X_{|\A|}=\A$. Let $$r(\bar{X})=(\rk(X_1)-\rk(X_0),\rk(X_2)-\rk(X_1),\dots , \rk(X_{|\A|})-\rk(X_{|\A|-1}))$$ be the rank vector of $\bar{X}$. Next choose $\{U_s\}$ as the basis (see \cite{D08} for details of this basis) for the ring of quasi-symmetric functions $\mathrm{Qsym}$ where $s$ runs over all the finite sequences of non-negative integers. The \emph{Derksen $\G$-invariant} of $M$ is \[\G (M)=\sum\limits_{\bar{X}}U_{r(\bar{X})}.\]

Derksen showed in \cite{D08} that his $\G$-invariant is valuative and that the Tutte polynomial can be derived from $\G$. Then in \cite{DF09} Fink and Derksen proved that Derksen's $\G$-invariant is universal among valuative invariants. Hence between the Tutte polynomial and Derksen's $\G$-invariant are the invariants which do not satisfy a deletion-contraction formula but are valuations. Many of these non-deletion-contraction yet valuative invariants, like the matroid Kazhdan-Lusztig polynomials of \cite{EPW-16} and the Billera-Jia-Reiner quasi-symmetric function of a matroid from \cite{BJR09} and the matroid volume polynomial of Eur in \cite{E20}, have not been calculated for some simple families of matroids. One reason is that the known recursions for these invariants are unwieldy. An exception is the result of Ferroni and Larson in \cite{FL-23} where they compute the Kazhdan-Lusztig polynomials for braid matroids using series parallel matroids. A central motivation for this work is to build structures that will help compute these invariants. So, one aim here is to develop a generalized deletion-contraction recursion to help compute valuative matroid invariants. The other main motivation is to build new finer invariants that may help distinguish certain types of matroids. 

The main focus of this note is the following generalization of Tutte polynomials.

\begin{definition}\label{Def-Tutte}

Let $M=(\A,\rk)$ be a polymatroid with ground set $\A$. For a positive integer $k$ let $$\C^k_\A =\left\{ (S_1,\dots ,S_k)\in \big(2^\A\big)^k \ | \ S_1\subseteq S_2\subseteq \cdots \subseteq S_k\right\}$$ and set $\C_{\A}^0=\{()\}$ where we are considering $()$ as the empty chain. The $k^{th}$ \emph{chain Whitney rank generating polynomial} of $M$ is $$W^k_M((a_i)_1^k;(b_i)_1^k)=\sum\limits_{(S_i)_1^k \in \C^k_\A } \prod\limits_{i=1}^k(a_i)^{\rk (M)-\rk (S_i)} (b_i)^{|S_i|-\rk(S_i)}.$$  Then the $k^{th}$ \emph{chain Tutte polynomial} of $M$ is 
$$T^k_M((x_i)_1^k;(y_i)_1^k)=W^k_M((x_i-1)_1^k;(y_i-1)_1^k).$$For convenience we may suppress the variables and just write $T^k_M$. Also, we will define $W^0_M=T_M^0=1$.
\end{definition}

This definition shows that $T^1_M(x_1;y_1)$ is the classic (poly)matroid Tutte polynomial. In section \ref{val-sec} we present the definition of this polynomial for any generalized permutahedron but for the majority of the paper we restrict our study to matroids.  Let $\mathrm{Mat} =\bigoplus \mathrm{Mat}_{r,n}$ be the collection of all matroids where $\mathrm{Mat}_{r,n}$ is the collection of matroids of rank $r$ on a ground set of $n$ elements. Then both the chain Whitney $W_M^k$ and Tutte $T_M^k$ polynomials are matroid invariants and we view them as functions \[W^k: \mathrm{Mat}_{r,n}\to \ZZ [a_1,\ldots ,a_n,b_1, \ldots ,b_n]\] and  \[T^k: \mathrm{Mat}_{r,n}\to \ZZ [x_1,\ldots ,x_n,y_1, \ldots ,y_n] .\] 

Many authors have studied generalizations of Tutte polynomials but they all seem to be different from that given in Definition \ref{Def-Tutte}. First, we note that our chain Tutte polynomials are far from the chain polynomials studied in \cite{RW-99} by Read and Whitehead and \cite{Tral-02} by Traldi. We do not consider weighted graphs or matroids like that of Zaslavsky in \cite{Zas-92}. Farr in \cite{Farr-93} considered a rank generating function for objects with rank functions defined over the entire real field (not just the non-negative integers). Cameron, Dinu, Micha{\l}ek, and Seynnaeve in \cite{CCMN22} defined a Tutte polynomial on a flag matroid which is a sequence of matroids. There are also generalizations to finer invariants. An important example in the multivariable Tutte polynomial (see \cite{Sokal05}) which is a complete invariant for matroids. There is also the work of Bernardi, K\'{a}lm\'{a}n, and Postnikov in \cite{OTP22} where they define one polynomial for each positive integer $n$ that parametrizes the Tutte polynomials of all polymatroids with a fixed ground set size $n$. Another perspective was given by Krajewski, Moffatt, and Tanasa in \cite{KMT-18} where a framework is outlined to build Tutte polynomials using Hopf algebras. Miezaki, Oura, Sakuma, and Shinohara in an announcement paper \cite{MOSS-19} define genus $g$ Tutte polynomials which are probably the most similar to the study here. However, the polynomials of Miezaki, Oura, Sakuma, and Shinohara have many more terms and are much finer invariants than ours. Actually the main result (Theorem 3.1 in \cite{MOSS-19}) of Miezaki, Oura, Sakuma, and Shinohara is that their collection of genus $g$ Tutte polynomials are complete invariants for matroids where ours will never be. The polynomials in \cite{KMT-18} seem close to ours but they are not built from sequences of ground set elements. The polynomials in  \cite{KMT-18} satisfy a classical ``2-term'' deletion-contraction formula instead of our $k$-term version. 

The first goal of this work is to present a generalized recursion for $T_M^k$. In order to state a kind of recursion we first need the following polynomials. These polynomials have three inputs: the matroid, a ground set element, and an integer for where the split occurs. In order to define these polynomials we need to recall some more matroid operations. For a matroid $M=(\A,\I)$ the \emph{deletion} of $M$ by $a\in \A$ is the matroid $M\bs S=(\A \bs a,\I')$ where $\I'=\{I\subseteq \A\bs a \ |\ I\in \I \}$. The \emph{contraction} of $M$ by $a\in \A$ is the matroid $M\bs a=(\A \bs a,\I'')$ where $$\I''=\left\{ \begin{array}{ll}
\{I\subseteq \A\bs S| I\cup a\in \I \} & \text{ if } a \text{ is not a loop}\\
\I & \text{ if } a \text{ is a loop.}
\end{array} \right._.$$ If $S\subseteq \A$ the deletion $M\bs S$ and the contraction $M/S$ are the matroids defined by repeatedly deleting and contracting by elements in $S$ respectively. The rank functions of these successive deletions and contractions satisfy $\rk_{M\bs S}(A)=\rk_M(A)$ and $\rk_{M/S}(A)=\rk_M(A\cup S)-\rk_M(S)$ for $A\subseteq \A\bs S$.

\begin{definition}\label{split-polys} Let $M$ be a matroid with atoms $\A$ and $k\in \Z_{\geq0}$. Then we define the $k^{th}$ \emph{split chain Tutte polynomial} of $M$ and $a\in \A$ \emph{at split term} $j$ where $0\leq j\leq k$ by $sT^{k,0}_{M,a}=T^k_{M/a}$ for the case $j=0$, $sT^{k,k}_{M,a}=T^k_{M\bs a}$ for the case $j=k$, and for $j\in \{1,\ldots ,k-1\}$ 

\begin{align*} sT^{k,j}_{M,a}((x_i),(y_i)) &=\sum\limits_{(S_i)_{1}^{j} \in \C^{j}_{\A -a} } \prod\limits_{i=1}^{j}(x_{i}-1)^{\rk(M\bs a)-\rk_{M\bs a}(S_i)} (y_{i}-1)^{|S_i|-\rk_{M\bs a}(S_i)}\\
&\hspace{1.5cm} \cdot \sum\limits_{\substack{(S_i)_{j+1}^{k} \in \C^{k-j}_{\A -a}\\ S_j\subseteq S_{j+1}} } \prod\limits_{i=j+1}^{k}(x_{i}-1)^{\rk(M/ a)-\rk_{M/ a}(S_i)} (y_{i}-1)^{|S_i|-\rk_{M/ a}(S_i)}.\end{align*} 

\end{definition}

We present another view of $sT^{k,j}_{M,a}((x_i),(y_i))$ in section \ref{def-sec}. We do not know of a chain Tutte polynomial minor decomposition of $sT^{k,j}_{M,a}((x_i),(y_i))$. Now we state the generalized deletion-contraction formula.

\begin{theorem}\label{recursion}

Let $M=(\A,\rk)$ be a matroid and $a\in \A$. 

\begin{enumerate}

\item If $a\in \A$ is a loop then $$T^k_M((x_i);(y_i))=T^k_{L}((x_i);(y_i))T^k_{M\bs a}((x_i);(y_i))$$ where $L$ is the matroid of a loop.

\item If $a\in \A$  is a coloop then $$T^k_M((x_i);(y_i))=T^k_{C}((x_i);(y_i))T^k_{M/ a}((x_i);(y_i))$$ where $C$ is the matroid of a coloop.

\item If $a\in \A$ is not a loop and not a coloop then $$T^k_M((x_i);(y_i))=\sum\limits_{j=0}^k sT^{k,j}_{M,a}((x_i);(y_i)).$$

\end{enumerate}

\end{theorem}

The next main result is relating the chain Tutte polynomials to Derksen's $\G$-invariant. We show that we can ``linearly'' determine Derksen's $\G$-invariant from the $|\A|^{th}$ chain Tutte polynomial. In order to properly define the homomorphism for the next theorem we have to consider the chain Whitney polynomial $W^k_M$ as an element in the infinite direct sum of the multivariable Laurent polynomial ring $\bigoplus_{r,n}\ZZ [a_i^{\pm 1},b_i^{\pm 1}]_{i \in [n]}$.

\begin{theorem}\label{T-universal}

There exists an additive homomorphism \[\Psi: \bigoplus_{r,n}\ZZ [a_i^{\pm 1},b_i^{\pm 1}]_{i \in [n]} \to \mathrm{Qsym}\] such that $\Psi (W_M^{|\A|})=\G (M)$ for any matroid $M$.

\end{theorem}

Since $T^1_M$ is the classic Tutte polynomial and $T^{|\A|}_M$ is essentially equivalent to Derksen's $\G$-invariant, the polynomials $T^k_M$ are a sequence of polynomials between, in terms of successive refinement, the classical Tutte polynomial and Derksen's $\G$-invariant. 

A starting place for this study was the work of Dupont, Fink, and Moci in \cite{DFM-19} where they define universal Tutte characters and their convolutions.  For this study we use the language and perspective in \cite{AS-20} where Ardila and Sanchez make the setting of Tutte characters concrete. In section \ref{val-sec} we define generalized permutahedra and valuations on matroids. Then we define a slightly different chain Tutte polynomial in the setting of generalized permutahedra by convolutions Tutte characters. Applying the tools of Ardila and Sanchez we prove these chain Tutte polynomials are valuations on generalized permutahedra. Then restricting this result to matroids we obtain the following (see section \ref{val-sec} for details).

\begin{theorem}\label{chainTutte-val}

The $k^{th}$ chain Whitney polynomial and chain Tutte polynomial are matroid valuations.

\end{theorem}

Since Derksen's $\G$-invariant is valuatively universal, which was proved by Fink and Derksen in \cite{DF09}, we get that the $|\A|^{th}$ chain Tutte polynomial is equivalent to $\G$. Hence we have proved that the chain Tutte polynomials have made a spectrum of finer invariants from the classical Tutte polynomial to the $\G$-invariant.

Another aim of this paper is to study some relatively unknown matroid invariants as evaluations of chain Tutte polynomials. First we look at some constant evaluations of the chain Tutte polynomials. While the constant evaluation formulas we present are not surprising, we include them to illustrate how chain Tutte polynomials generalize classic formulas for the Tutte polynomial. Then we consider some polynomial evaluations of chain Tutte polynomials. The \emph{M\"obius polynomial} of a matroid $M$ is \begin{equation}\label{Mobdef}\bar{\chi}_M(s,t)=\sum\limits_{X\leq Y\in L(M)}\gm (X,Y)s^{\rk(M)-\rk(X)}t^{\rk (M)-\rk(Y)}\end{equation} where $\mu$ is the M\"obius function defined in section \ref{prelim}. It turns out one can obtain the M\"obius polynomial of a matroid from the second chain Tutte polynomial and together with Theorem \ref{chainTutte-val} we get the following.

\begin{theorem}\label{evalT2-mob}

If $M$ is a matroid then $$T^2_{M}(1-s,1-t;0,0)=\bar{\chi}_M(s,t)$$ and $\bar{\chi}_M$ is a matroid valuation.

\end{theorem}

The M\"obius polynomial has seen some recent activity in coding theory (see \cite{JV-21} and \cite{Jur-12}) and some applications in non-commutative algebras defined by layered graphs in \cite{RW-09} by Retakh and Wilson which originated from work of Gelfand, Retakh, Serconek, and Wilson in \cite{GRSW-05}.  Using the generalized deletion-contraction formula applied to the evaluation formulas from chain Tutte polynomials we get a new recursion for the M\"obius polynomial (Corollary \ref{Mob-recur}). 

We conclude our study by applying the chain Tutte polynomial to a few other related polynomials. We study the opposite characteristic polynomial of the lattice of flats of a matroid. This opposite characteristic polynomial is also an easy evaluation of the second chain Tutte polynomial. The reason to study the opposite characteristic polynomial is that in \cite{JW-20} Johnson and the author presented a problem to prove that the generalized M\"obius function defined using the $J$-function from \cite{JW-20} is a valuation. We prove this result by proving that the opposite characteristic polynomial is a valuation since it is an evaluation of the second chain Tutte polynomial. Next we examine the  expected codimension of a matroid variety defined by Ford in \cite{Ford-15}.  We show that the expected codimension is determined by the second chain Tutte polynomial and then apply our inductive result to obtain a recursion for the expected codimension.

We organize this paper as follows. In section \ref{prelim} we collect some classic facts about matroids, generalized permutahedra, M\"obius functions, Tutte polynomials and Derksen's $\G$-invariant. Then in section \ref{def-sec} we present basic results about chain Tutte polynomials, including Theorem \ref{recursion}, which gives a generalized deletion-contraction formula. Next in section \ref{val-sec} we define chain Tutte polynomials on generalized permutahedra and prove that chain Tutte polynomials are valuations. We end with section \ref{evals-sec} where we examine various applications of chain Tutte polynomials by computing specific polynomial evaluations. There we study the M\"obius polynomial, the opposite characteristic polynomial, the generalized $J$-M\"obius polynomial, and the expected codimension of a matroid variety.

\

\noindent {\bf Acknowledgments:} Special thanks go the the referee for many suggestions which greatly improved the article. The author is very thankful to Cl\'ement Dupont for teaching him about universal Tutte characters, Graham Denham for some conversations about Tutte polynomials, Carolyn Chun for insights on matroid theory, Franklin Kenter for some comments in graph theory, Will Traves for many helpful discussions, and Joseph Bonin for helpful comments on matroid theory and earlier versions of this note. The author is very thankful to Alex Fink for many helpful comments and mathematical suggestions for an earlier version of this article. In particular, Fink made the suggestion to find the expected codimension with the chain Tutte polynomials.


\section{Preliminaries}\label{prelim}

In this section we review some definitions and results as well as some reformulations of basic properties of well known invariants on matroids and lattices. For matroid terminology we follow \cite{Ardila-15} and \cite{Oxley}.

\subsection{Posets and polymatroids} In a finite poset $L$ a maximal chain is a strictly ordered sequence of elements $a_1<a_2<\cdots <a_k$ such that there does not exist $b\in L$ and $1\leq i\leq k-1$ with $a_i<b<a_{i+1}$. A finite poset $L$ is a ranked lattice if all least upper bounds and greatest lower bounds exist and the lengths of all maximal chains are the same. If $L$ is a ranked lattice then the least upper bound is called a join and denoted by $\vee$ and the greatest lower bound is called a meet and denoted by $\wedge$. If a finite lattice $L$ is ranked then it has a rank function usually denoted by $\rk : L\to \Z_{\geq 0}$ where $\rk(X)$ is the length of a maximal chain from the least element to $X$. An atom in a lower bounded lattice with bottom element $\hat{0}$ is any element $a\in L$ such that $a>\hat{0}$ and there does not exist $b\in L$ such that $\hat{0}<b<a$. We say a lower bounded lattice is atomic if every element is a join of atoms. Let $L$ be a finite ranked and atomic lattice with atoms $\A$ and rank function $\rk$. If $S\subseteq \A$ then we define the rank of $S$ as $\rk(\bigvee S)$.

A \emph{flat} of a matroid $M$ is a set $F\subseteq \A$ such that for all $x\in \A - F$, $\rk(F\cup \{x\})>\rk(F)$. The \emph{lattice of flats} of $M$, denoted by $L(M)$, is the set of flats of $M$ ordered by inclusion. The lattice of flats is a geometric lattice; in particular it is a ranked atomic lattice. We often work with $M$ through $L(M)$. When $M$ is simple, the atoms of $L(M)$ are the elements of $\A$. In general, the atoms of $L(M)$ are the parallelism classes of non-loops in $\A$.

Note that from an undirected graph $G=(V,E)$ with vertices $V$ and edges $E$ one can construct a matroid $M=(E,\rk)$ where the rank function is defined by $\rk(S)=|S|-c$ where $c$ is the number of connected components of the induced graph on the vertices $S$. We allow graphs to have multiedges and loops. In Section \ref{def-sec} we will consider some examples defined by graphs.

\subsection{M\"obius function} The M\"obius function of a locally finite poset $\calP$ is a function $\gm :\calP \times \calP \to \ZZ$ defined by setting $\gm (X,X)=1$ and $$\sum\limits_{X\leq Y\leq Z}\gm (X,Y)=0$$ for all $X,Z\in \calP$. First we recall a classic result on M\"obius functions on matroids by which we mean the M\"obius function on the lattice of flats $L(M)$ of the matroid $M$. For any subset of atoms $X\subseteq \A$ we write $\bigvee X$ for the join of these atoms inside the lattice of flats (a.k.a. the flat closure of $X$ in $L(M)$).

\begin{lemma}[\cite{Zas87}, Proposition 7.1.4]\label{m1} Let $M$ be a simple matroid with minimum flat $\hat{0}=\emptyset$ and maximum flat $\hat{1}=\A$. Then $$\gm (\hat{0},\hat{1})=\sum\limits_{\substack{X\subseteq \A \\ \bigvee X=\hat{1}}}(-1)^{|X|}.$$
\end{lemma}

For our purposes we need to extend this result slightly. Surely this is in the literature somewhere, but the author could not find this particular formulation so we include the proof.

\begin{lemma}\label{mobius-Whitneythm}

Let $M$ be a matroid with ground set $\A$ and $L(M)$ its lattice of flats. For all $X,Y\in L(M)$, $$\gm(X,Y)=\sum\limits_{\substack{A \subseteq B \subseteq \A\\ \bigvee A=X \\ \bigvee B=Y}}(-1)^{|A|+|B|}.$$

\end{lemma}

\begin{proof}

For $X,Y\in L(M)$ we define a function \[g(X,Y)=\sum\limits_{\substack{A \subseteq B \subseteq \A\\ \bigvee A=X \\ \bigvee B=Y}}(-1)^{|A|+|B|}\] and note here that $A,B\subseteq \A$ are not necessarily flats. Then for any fixed flats $X,Z\in L(M)$
\begin{align*}
\sum\limits_{\substack{Y\in L(M)\\ X\leq Y\leq Z}}g(X,Y)=& \sum\limits_{\substack{Y\in L(M)\\ X\leq Y\leq Z}} \left[ \sum\limits_{\substack{A \subseteq B \subseteq \A\\ \bigvee A=X \\ \bigvee B=Y}}(-1)^{|A|+|B|}\right]\\
=&\sum\limits_{\substack{Y\in L(M)\\ X\leq Y\leq Z}} \left[  \sum\limits_{\substack{A \subseteq \A\\ \bigvee A=X}}(-1)^{|A|}\left[ \sum\limits_{\substack{B \supseteq A \\ \bigvee B=Y}}(-1)^{|B|}\right]\right] \\
=& \sum\limits_{\substack{A \subseteq \A\\ \bigvee A=X}}(-1)^{|A|}\left[ \sum\limits_{\substack{Y\in L(M)\\ X\leq Y\leq Z}} \left[  \sum\limits_{\substack{B \supseteq A \\ \bigvee B=Y}}(-1)^{|B|}\right]\right] \\
=& \sum\limits_{\substack{A \subseteq \A\\ \bigvee A=X}}(-1)^{|A|}\left[ \sum\limits_{\substack{B \subseteq \A \\ A\subseteq B\subseteq Z}}(-1)^{|B|}\right] .\\
\end{align*}
The sum in the brackets in the last line above is zero in all cases except when $A=Z$ which implies that $X=Z$ and the entire expression is $(-1)^{2|Z|}=1$. Hence we have shown that the function $g$ satisfies the same recursion as the M\"obius function $\gm$.\end{proof}

\subsection{Characteristic polynomials and deletion-contraction} The characteristic polynomial has been one of the most studied matroid invariants. This invariant will appear in various places in this work.

\begin{definition}

Let $M=(\A,\rk)$ be a polymatroid and $L(M)$ its lattice of flats. The \emph{characteristic polynomial} of $M$ is $$\chi_M(t)=\sum\limits_{X\in L(M)} \gm (\hat{0},X)t^{\crk (X)}$$ where $\crk (X)=\rk(M)-\rk(X)$ (we sometime use this notation to condense exponents when needed).

\end{definition}

A key fact is that the characteristic polynomial of a matroid is an evaluation of the Tutte polynomial \[\chi_M(t)=(-1)^{\rk(M)}T_M(1-t;0).\] Now we recall the functions that satisfy a deletion-contraction formula.

\begin{definition}\label{gen-Tutte-groth}

Let $R$ be a commutative ring. We say that a function $f: \mathrm{Mat} \to R$ is a \emph{generalized Tutte-Grothendieck invariant} if for any $M\in \mathrm{Mat}$ with ground set $\A$ and $e\in \A$,

$$f(M)=\left\{\begin{array}{lll}
af(M\bs e)+bf(M/e) &\ \ &\text{if }e\text{ is neither a loop nor a coloop}\\
f(M\bs e)f(L) &&\text{if } e \text{ is a loop}\\
f(M/ e)f(C) &&\text{if } e \text{ is a coloop}\\
\end{array}\right.$$ for some fixed non-zero constants $a,b\in R$ where $C$ is the matroid of one coloop and $L$ is a matroid consisting of one loop.

\end{definition}

\section{Basic properties of chain Tutte polynomials}\label{def-sec}

For the majority of this note we focus primarily on the chain generalization of the classical Tutte polynomial given in Definition \ref{Def-Tutte}. In section \ref{val-sec} we also study an equivalent version (a change of coordinates) in Definition \ref{GP-uniTutte-def}.

For most of this study we will focus on the setting of matroids. However, we may examine the case of a more general ranked, atomic lattice $L$ in which case we write $T_L^k((x_i);(y_i))$ where $\rk (S_i)$ means the rank of $\rk(\bigvee S_i)$ in $L$. In this case the atoms are the ground set elements and for a subset of ground set elements $S_i$ the rank $\rk(S_i)$ is the usual matroid rank function.

\subsection{Formulas for $T^k$ in terms of $T^{|\A|}$} As desired, the higher chain Tutte polynomials determine all the lower chain Tutte polynomials.

\begin{lemma}\label{getT1}

For any matroid $M$ and any $k\geq 1$ $$T^{k+1}_M(2,2x_1-1,x_2,\dots ,x_k;2,2^{-1}y_1+2^{-1},y_2,\dots ,y_k)=2^{\rk (M)}T^k_M(x_1,\dots ,x_k;y_1,\dots ,y_k).$$

\end{lemma}

\begin{proof} Computing the evaluation we get

$$T^{k+1}_M(2,2x_1-1,x_2,\dots ,x_k;2,2^{-1}y_1+2^{-1},y_2,\dots ,y_k)$$

$$=\sum\limits_{(S_i)_1^{k+1} \in \C^{k+1}_\A } 2^{\rk(M)-|S_2|}\prod\limits_{i=1}^k(x_i-1)^{\rk (M)-\rk (S_{i+1})} (y_i-1)^{|S_{i+1}|-\rk(S_{i+1})}$$ 

$$=2^{\rk(M)}\sum\limits_{(S_i)_2^{k+1} \in \C^{k}_\A } \prod\limits_{i=1}^k(x_i-1)^{\rk (M)-\rk (S_{i+1})} (y_i-1)^{|S_{i+1}|-\rk(S_{i+1})}$$ since the number of subsets $S_1$ of $S_2$ is exactly $2^{|S_2|}$. By noting the last quantity above is just a reindexing of $2^{\rk(M)}T^k_M(x_1,\dots ,x_k;y_1,\dots ,y_k)$ we have completed the proof. \end{proof}

Hence the classic Tutte polynomial $T^1$ is determined by any higher chain Tutte polynomial $T^k$ for $k\geq 2$. Moreover, we have the following which gives that the $|\A|^{th}$ chain Tutte polynomial $T^{|\A|}_M$ for $M$ determines all the other chain Tutte polynomials $T^k_M$. 

\begin{proposition}\label{Adeter} 

For all $k\geq 1$ there is a homomorphism $\gs_k: \ZZ [x_1,\dots ,x_{|\A|},y_1,\dots , y_{|\A|}] \to \ZZ [x_1,\dots ,x_k,y_1,\dots ,y_k]$ such that $\gs_k(T^{|\A|}_M)=T^k_M$ for all matroids $M$.

\end{proposition}

\begin{proof}

For $k\leq |\A|$ Lemma \ref{getT1} gives the result. For $k>|\A |$ any term of $T^k_M$ will have at least $k-|\A|$ of the same consecutive exponents on both the $x_i$ and the $y_i$ of the form $x_i^ax_{i+1}^a$ and $y_i^cy_{i+1}^c$. Let $D_k=\{(b_1,\dots ,b_{|\A|})|b_i\in \NN, \ b_1+\cdots +b_{|\A|}=k-|\A|\}$ be the set of distributions for determining new terms from $T^{|\A|}_M$ to $T^k_M$. Also, define $B_1=0$ and for any $2\leq j\leq |\A|$ set $B_j=b_1+b_2+\cdots +b_{j-1}$. Denote each term of $T^{|\A|}_M$ determined by a chain $(S_i)_1^{|\A|}$ by \[tm(S_i)_1^{|\A|}:=\prod\limits_{i=1}^{|\A|} (x_i-1)^{\rk(M)-\rk(S_i)}(y_i-1)^{|S_i|-\rk(S_i)} .\] Then we define \[ \gs_k(tm(S_i)_1^{|\A|})=\sum\limits_{(b_1,\dots , b_{|\A|})\in D_k} \prod\limits_{i=1}^{|\A|} \prod\limits_{m=0}^{b_i}(x_{i+B_j+m}-1)^{\rk(M)-\rk(S_i)}(y_{i+B_j+m}-1)^{|S_i|-\rk(S_i)} .\] Then extending $\gs_k$ linearly and applying to all the terms of $T^{|\A|}_M$ we get $\gs_k (T^{|\A|}_M) =T^k_M$. \end{proof}


\subsection{Derksen's $G$-invariant}\label{G-inv-sub}

In \cite{DF09} Derksen and Fink proved that the $G$-invariant is universal among valuative invariants.

\begin{theorem}[{\cite[Theorem 1.4]{DF09}}]\label{DF-universal}

The coefficients of $\G$ span the vector space of all valuative matroid invariants with values in $\mathbb{Q}$.
\end{theorem}

%
%
%
We juxtapose this with a formula for determining Derksen's $\G$-invariant in terms of $T^{|\A|}$.

\begin{proof}[Proof of Theorem \ref{T-universal}]

First we fix the rank to be $r\geq 0$ and the number of ground set elements to be $n=|\A|$. Then we consider the function \[\psi_{r,n} : \ZZ [a_1^{\pm 1},\ldots ,a_{n}^{\pm 1},b_1^{\pm 1},\ldots b_{n}^{\pm 1}] \to  \ZZ [a_1^{\pm 1},\ldots ,a_{n}^{\pm 1},b_1^{\pm 1},\ldots b_{n}^{\pm 1}]\] defined by $\psi_{r,n}(1)=1$,

\[ a_i^{\pm 1} \mapsto \left\{ \begin{array}{cc}
(a_i^{-1}a_{i+1}b_i^{-1})^{\pm 1} & \text { if } 1\leq i<n\\
(a_{n}^{-1}b_n)^{\pm 1} & \text { if }  i=n\\
\end{array} \right. \]
and \[ b_i^{\pm 1}\mapsto b_i^{\pm 1}\]
for all $1\leq i\leq n$ and extend linearly. Also, define \[s_{r,n} : \ZZ [a_1^{\pm 1},\ldots ,a_{n}^{\pm 1},b_1^{\pm 1},\ldots b_{n}^{\pm 1}] \to  \ZZ [a_1^{\pm 1},\ldots ,a_{n}^{\pm 1},b_1^{\pm 1},\ldots b_{n}^{\pm 1}]\] by \[ s_{r,n}(p)= p\cdot  a_1^{r}\prod\limits_{i=1}^nb_i^r.\] 

Then composing these two maps we get \[ (s_{r,n} \circ \psi_{r,n})(W_M^n)= \sum\limits_{(S_i)\in \C^n_{\A}} \prod\limits_{i=1}^n a_i^{\rk(S_i)-\rk(S_{i-1})}b_i^{|S_i|} .\]

Consider the decomposition $\mathrm{QSym}=\bigoplus_{r,n}\mathrm{QSym}_{r,n}$ where $\mathrm{QSym}_{r,n}$ are the quasi-symmetric functions spanned by $U_s$ where $s$ has length $n$ and the sum of the elements in $s$ is $r$. With this decomposition we can define our final map \[ \theta_{r,n} : \ZZ [a_1^{\pm 1},\ldots ,a_{n}^{\pm 1},b_1^{\pm 1},\ldots b_{n}^{\pm 1}] \to \mathrm{QSym}_{r,n}\] on monomials by \[ \prod\limits_{i=1}^na_i^{e_i}b_i^{f_i}\mapsto \left\{ \begin{array}{cc} U_{(e_1,\ldots , e_n)}& \text{ if } \forall i, \ f_i=i,\ \text{ and } \sum e_i=r\\
0 & \text{ else}\\
\end{array}\right. .\]
By construction for any polymatroid $M\in \mathrm{Mat}_{r,n}$ we have $(\theta_{r,n}\circ s_{r,n}\circ \psi_{r,n})(W_M^n)=\G (M).$ To finish the proof we consider the extension \[ \Psi =\oplus_{r,n} (\theta_{r,n}\circ s_{r,n}\circ \psi_{r,n})\] on all of \[ \bigoplus_{r,n}  \ZZ [a_1^{\pm 1},\ldots ,a_{n}^{\pm 1},b_1^{\pm 1},\ldots b_{n}^{\pm 1}] \to \mathrm{QSym}\] to get the desired result that $\Psi(W_M^n)=\G(M)$ for any polymatroid.  \end{proof}

\subsection{Direct sums and duals}

Next we present a few basic facts about these polynomials which generalize the properties of the Tutte polynomial $T^1_M$. First we show that up to a change of coordinates a matroid and its dual have the same the chain Tutte polynomial.

\begin{proposition}\label{Tutte-dual}

If $M^*$ is the dual matroid of $M$ then $$T^k_{M^*}(x_1,\dots ,x_k;y_1,\dots , y_k)=T^k_M(y_k,\dots ,y_1;x_k,\dots ,x_1).$$

\end{proposition}

\begin{proof}

First we note that the map sending a flag $(A_1,\dots ,A_k)$ to its complement $(\A-A_k,\dots,\A-A_1)$ is a permutation of $\C^k(\A)$. Then just use Definition \ref{Def-Tutte} and the fact that $\rk_{M^*}(X)=|X|-\rk_{M}(\A)+\rk_{M}(\A-X)$.\end{proof}

The next fact is that $T^k$ is multiplicative over matroid direct sum. Recall that if $M_1=(\A_1,\I_1)$ and $M_2=(\A_2,\I_2)$ are matroids then $M_1\oplus M_2=(\A_1 \sqcup \A_2,\{I_1\sqcup I_2\ |\ I_1\in \I_1,I_2\in \I_2\})$. 

\begin{proposition}\label{Tprod}

If $M_1=(\A_1,\I_1)$ and $M_2=(\A_2,\I_2)$ are matroids then $$T^k_{M_1\oplus M_2}((x_i);(y_i))=T^k_{M_1}((x_i);(y_i))\cdot T^k_{M_2}((x_i);(y_i)).$$

\end{proposition}

\begin{proof}

The proof is done by decomposing the definition of $T^k_{M_1\oplus M_2}((x_i);(y_i))$ in terms of the disjoint sets $\A_1$ and $\A_2$. We find that 

\

$T^k_{M_1\oplus M_2}((x_i);(y_i))=$
\begin{align*}
=& \sum\limits_{(S_i) \in \C^k_\A } \prod\limits_{i=1}^k(x_i-1)^{\rk (\A_1\sqcup \A_2)-\rk (S_i)} (y_i-1)^{|S_i|-\rk(S_i)}\\
=&\sum\limits_{(S_i^1) \in \C^k_{\A_1} } \sum\limits_{(S_i^2) \in \C^k_{\A_2} }\prod\limits_{i=1}^k(x_i-1)^{\rk (\A_1\sqcup \A_2)-\rk (S_i^1\sqcup S_i^2)} (y_i-1)^{|S_i^1\sqcup S_i^2|-\rk(S_i^1\sqcup S_i^2)}\\
=&\sum\limits_{(S_i^1) \in \C^k_{\A_1} } \sum\limits_{(S_i^2) \in \C^k_{\A_2} }\prod\limits_{i=1}^k(x_i-1)^{\rk (\A_1)+\rk(\A_2)-\rk (S_i^1)-\rk( S_i^2)} (y_i-1)^{|S_i^1|+|S_i^2|-\rk(S_i^1)-\rk( S_i^2)}\\
=&\sum\limits_{(S_i^1) \in \C^k_{\A_1} } \prod\limits_{i=1}^k(x_i-1)^{\rk (\A_1)-\rk (4S_i^1)} (y_i-1)^{|S_i^1|-\rk(S_i^1)}\\
& \hspace{1cm} \cdot \sum\limits_{(S_i^2) \in \C^k_{\A_2} }\prod\limits_{i=1}^k(x_i-1)^{\rk(\A_2)-\rk( S_i^2)}(y_i-1)^{|S_i^2|-\rk( S_i^2)}\\
\end{align*} which is exactly the product $T^k_{M_1}((x_i);(y_i))\cdot T^k_{M_2}((x_i);(y_i)).$ \end{proof}

\subsection{Basic Examples} Through some examples we see that chain Tutte polynomials can have negative coefficients, but have more information than the classic Tutte polynomial.

\begin{example}\label{boolean-ex} 

Let $\B_1$ be the rank 1 Boolean lattice (subset lattice of a set with 1 element) also known as the free matroid on one element or $U_{1,1}$ or the matroid of one coloop $C$. Then by direct computation $$T^k_{\B_1}((x_i);(y_i))=1+\sum\limits_{i=1}^k \prod\limits_{j=1}^{i} (x_j-1).$$ Notice that even $T^2_{\B_1}(x_1,x_2;y_1,y_2)$ has negative coefficients. Then by Theorem \ref{Tprod} we have that 
$$T^k_{\B_n}((x_i);(y_i))=\left(1+\sum\limits_{i=1}^k \prod\limits_{j=1}^{i} (x_j-1)\right)^n$$ where $\B_n$ is the free matroid on $n$ elements (aka Boolean subset lattice or the uniform matroid $U_{n,n}$).\end{example}

\begin{figure}
\begin{tikzpicture}
\draw[thick] (0,0)--(2,0)--(2,2)--(0,0)--(0,2)--(2,2);
\draw[thick] (1,1)--(0,2);
\draw[thick] (2,2)--(1,3);

\draw[thick] (0,2) to[bend right] (1,3);
\draw[thick] (0,2) to[bend left] (1,3);

\filldraw[black] (0,0) circle (3pt);
\filldraw[black] (1,1) circle (3pt);
\filldraw[black] (1,3) circle (3pt);

\filldraw[black] (0,2) circle (3pt);

\filldraw[black] (2,0) circle (3pt);

\filldraw[black] (2,2) circle (3pt);

\node at (-1.1,1) {$G_1$};


\draw[thick] (4,0)--(6,0)--(6,2)--(4,2)--(4,0)--(4,2)--(6,2);
\draw[thick] (4,2)--(5,3)--(6,2);

\draw[thick] (4,0)--(5,1)--(6,0);

\draw[thick] (4,2) to[bend right] (5,1);
\draw[thick] (4,2) to[bend left] (5,1);

\filldraw[black] (4,0) circle (3pt);

\filldraw[black] (4,2) circle (3pt);

\filldraw[black] (6,0) circle (3pt);

\filldraw[black] (6,2) circle (3pt);

\filldraw[black] (5,3) circle (3pt);

\filldraw[black] (5,1) circle (3pt);

\node at (7.1,1) {$G_2$};

\end{tikzpicture}
\caption{The Gray graphs}\label{fig2gray}
\end{figure}
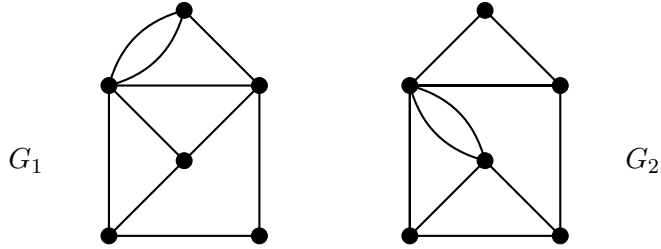

Now we show that the second chain Tutte polynomial $T^2$ has more information than the Tutte polynomial $T^1$ for a few classic examples.

\begin{example}

The two Gray graphs $G_1$ and $G_2$, pictured in Figure \ref{fig2gray} (see \cite[Example 3.4]{D08} and \cite{Bryl-81}), have the same Tutte polynomial even though they are not isomorphic graphs. However, the second chain Tutte polynomials differ for $G_1$ and $G_2$. For example, the coefficient of $y_1y_2^3$ in $T^2_{G_1}$ is 1 but in  $T^2_{G_1}$ it is 0. These polynomials both have 312 non-zero terms hence we do not list all coefficients (these computations were made using Sage \cite{sage}). 

\end{example}

\begin{example}

\begin{figure}
\begin{tikzpicture}
\draw[thick] (0,0)--(2,0)--(0,1)--(0,2)--(1,0)--(0,0)--(0,1);

\filldraw[black] (-.07,.07) circle (3pt);

\filldraw[black] (.07,-.07) circle (3pt);

\filldraw[black] (0,1) circle (3pt);

\filldraw[black] (0,2) circle (3pt);

\filldraw[black] (1,0) circle (3pt);

\filldraw[black] (2,0) circle (3pt);

\filldraw[black] (.66,.66) circle (3pt);

\node at (-.7,1) {$M_1$};


\draw[thick] (4,2)--(4,0)--(7,0);

\filldraw[black] (4,1) circle (3pt);

\filldraw[black] (4.1,2) circle (3pt);

\filldraw[black] (3.9,2) circle (3pt);

\filldraw[black] (5,0) circle (3pt);

\filldraw[black] (7,0) circle (3pt);

\filldraw[black] (4,0) circle (3pt);

\filldraw[black] (6,0) circle (3pt);

\node at (7.8,1) {$M_2$};

\end{tikzpicture}
\caption{Two non-isomorphic matroids with same Tutte polynomial $T^1$ but different chain Tutte polynomial $T^2$}\label{fig2mat}
\end{figure}
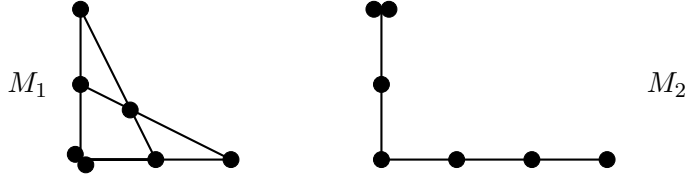

Let $M_1$ and $M_2$ be the rank 3 matroids on 7 elements in Figure \ref{fig2mat} (\cite[Example 6.2.18]{BO92}). The coefficient of $a_1^2a_2b_1b_2^2$ in $W^2_{M_1}$ is 2, but the coefficient of the same monomial in $W^2_{M_2}$ is 1. However, the classic Tutte polynomial of these two matroids is the same. The Derksen invariant $\G$ also distinguishes these two matroids (\cite[Example 3.6]{D08}).

\end{example}

We will see below in Section \ref{val-subsec} that Derksen's $\G$-invariant determines $T^k$ for any $k>0$. Hence, also the chain Tutte polynomials $T^k$ for all $k$ will not distinguish the examples built by Bonin in \cite{Bon-22} and Bonin and Long in \cite{BL-22}. 

\subsection{Recursion for chain Tutte polynomials}

First we look at a few examples to show that $T^k$ for $k>1$ is not a generalized Tutte-Grothendieck invariant.

\begin{example}\label{not-Tutte-inv} For $r,n\in \NN$ and $r\leq n$ let $U_{r,n}$ be the matroid with ground set $[n]$ and independent sets are all subsets of size less than or equal to $r$ of $[n]$; $U_{r,n}$ are called uniform matroids. In order to show a deletion-contraction recursion does not exist we study the matroids $U_{2,4}$, $U_{2,3}$, and $U_{1,2}$. The deletions and contractions independently of the choice of ground set element $e$ of these matroids are as follows: $U_{2,4}\bs e=U_{2,3}$, $U_{2,4}/e=U_{1,3}$, $U_{2,3}\bs e=U_{2,2}$, $U_{2,3}/ e=U_{1,2}$, $U_{1,2}\bs e=U_{1,1}$ and $U_{1,2}/e=U_{0,1}$. We can compute $T^2_{U_{0,1}}=(x_1-1)(x_2-1)+(x_1-1)(x_2-1)(y_2-1)+(x_1-1)(x_2-1)(y_1-1)(y_2-1)$, $T^2_{U_{1,1}}=1+ (x_1-1)(x_2-1)+(x_1-1)$, and $$T^2_{U_{1,2}}=(x_1-1)(x_2-1)+2(x_1-1)+(x_1-1)(y_2-1)+2(y_2-1)+(y_1-1)(y_2-1)+2.$$ Then using the computer algebra system Sage \cite{sage} we get $$T^2_{U_{2,3}}=x_1^2x_2^2 + x_1^2x_2 - 2x_1x_2^2 + x_1^2y_2 + x_1x_2 + x_2^2 + x_1y_2 + y_1y_2 - 2x_2 - y_1 + 1$$ and $$T^2_{U_{2,4}}=x_1^2x_2^2 + x_1^2y_2^2 + y_1^2y_2^2 + 2x_1^2x_2 - 2x_1x_2^2 + 2x_1^2y_2 - 2y_1^2y_2 + 2x_1y_2^2 + 2y_1y_2^2 + x_2^2 + y_1^2 - 2x_2 - 2y_1 + 1.$$

If there was a deletion-contraction recursion from Definition \ref{gen-Tutte-groth} then we would have the following equations $T^2_{U_{1,2}}=aT^2_{U_{1,1}}+bT^2_{U_{0,1}}$, $T^2_{U_{2,3}}=aT^2_{U_{2,2}}+bT^2_{U_{1,2}}$, and $T^2_{U_{2,4}}=aT^2_{U_{2,3}}+bT^2_{U_{1,3}}.$ Hence, elements in the kernel of the matrix \[A=\left[\begin{array}{ccc}
T^2_{U_{1,1}}&T^2_{U_{0,1}}&T^2_{U_{1,2}}\\
T^2_{U_{2,2}}&T^2_{U_{1,2}}&T^2_{U_{2,3}}\\
T^2_{U_{2,3}}&T^2_{U_{1,3}}&T^2_{U_{2,4}}\\
\end{array}\right] \] over the field of rational functions would be a candidates for $a$ and $b$. However, the determinant of $A$ is a non-zero polynomial of degree 10 (again computed using Sage). Hence it's not possible for $T^2$ to be a generalized Tutte-Grothendieck invariant.  Since $T^2$ is an evaluation of $T^k$ for all $k>1$ we know that these are also not generalized Tutte-Grothendieck invariants.

\end{example}

Next we study the split chain Tutte polynomials in Definition \ref{split-polys}. Note that these split chain Tutte polynomials are just reorderings of the original chain Tutte polynomials separated by terms containing or not containing the element $a\in \A$. Gathering terms in order to present a formula for the split chain Tutte polynomials we have 

\[sT^{k,j}_{M,a}((x_i),(y_i))=\] \[\sum\limits_{(S_i)_{1}^{k-j} \in \C^{k-j}_{\A -a} }T^j_{M|{S_{1}}}\prod\limits_{i=1}^{k-j}(x_{i+j}-1)^{\rk(M/a)-\rk_{M/a}(S_i)} (y_{i+j}-1)^{|S_i|-\rk_{M/a}(S_i)}\prod\limits_{i=1}^j(x_i-1)^{\rk (M/S_{1})} .\]



By construction for $a\in \A$ not a loop and not a coloop we can state the classical recursion for the Tutte polynomial as $$T^1_M=sT^{1,0}_{M,a}+sT^{1,1}_{M,a}.$$ Now Theorem \ref{recursion} generalizes this for $k>1$ which we now prove.

\begin{proof}[Proof of Theorem \ref{recursion}]

We split the terms of $T^k_M$ up by which sets contain $a\in \A$. For $j\in \{1,\ldots ,k-1\}$ let $$st_{M,a}^{k,0}((x_i);(y_i))=\sum\limits_{\substack{(S_i)_{1}^k \in \C^{k}_{\A}\\ a\in S_{1}} } \prod\limits_{i=1}^k(x_i-1)^{\rk (M)-\rk (S_i)} (y_i-1)^{|S_i|-\rk(S_i)},$$ $$st_{M,a}^{k,k}((x_i);(y_i))=\sum\limits_{(S_i)_{1}^k \in \C^{k}_{\A- a} } \prod\limits_{i=1}^k(x_i-1)^{\rk (M)-\rk (S_i)} (y_i-1)^{|S_i|-\rk(S_i)},$$ and $$st_{M,a}^{k,j}((x_i);(y_i))= \sum\limits_{\substack{(S_i)_{j+1}^k \in \C^{k-j}_{\A}\\ a\in S_{j+1}} }\sum\limits_{(S_i)_1^j \in \C^j_{S_{j+1}- a} }  \prod\limits_{i=1}^k(x_i-1)^{\rk (M)-\rk (S_i)} (y_i-1)^{|S_i|-\rk(S_i)}. $$ Then by construction $$T_M^k=\sum\limits_{j=0}^k st_{M,a}^{k,j}((x_i);(y_i)).$$ 

We know that for  all $S\subseteq \A$ with $a\in S$ we have $\rk_M(S)= \rk_{M/a}(S-a) +\rk_M(a)$. Since $a\in \A$ is not a loop for all $S\subseteq \A$ with $a\in S$ we have $|S|-\rk(S)=|S-a|+1-(\rk_{M/a}(S-a)+\rk_M(a)) =|S-a|-\rk_{M/a}(S-a)$. Also, for $S\subseteq \A$ with $a\in S$ we have $\rk(M)-\rk_M(S)=(\rk_{M/a}(\A-a) +\rk_M(a))-(\rk_{M/a}(S-a)+\rk_M(a))=\rk(M/a)-\rk_{M/a}(S-a)$. Putting this into the $j=0$ term we have $$st_{M,a}^{k,0}((x_i);(y_i))=\sum\limits_{\substack{(S_i)_{1}^k \in \C^{k}_{\A}\\ a\in S_{1}} } \prod\limits_{i=1}^k(x_i-1)^{\rk(M/a)-\rk_{M/a}(S_i-a)} (y_i-1)^{|S_i-a|-\rk_{M/a}(S_i-a)}$$ $$=\sum\limits_{(S_i)_{1}^k \in \C^{k}_{\A-a}} \prod\limits_{i=1}^k(x_i-1)^{\rk(M/a)-\rk_{M/a}(S_i)} (y_i-1)^{|S_i|-\rk_{M/a}(S_i)}=T^k_{M/a}=sT^{k,0}_{M,a}.$$

In the $j=k$ case the argument is much simpler. Since $a\in A$ is not a coloop we know that $rk(M)=\rk(M\bs a)$. So, $$st^{k,k}_{M,a}=\sum\limits_{(S_i)_{1}^k \in \C^{k}_{\A- a} } \prod\limits_{i=1}^k(x_i-1)^{\rk (M\bs a)-\rk_{M\bs a} (S_i)} (y_i-1)^{|S_i|-\rk_{M\bs a}(S_i)}=T^k_{M\bs a}=sT^{k,k}_{M,a}.$$

Now for any $j\in \{1,\ldots ,k-1\}$ we see that $$st_{M,a}^{k,j}((x_i);(y_i))=\sum\limits_{\substack{(S_i)_{j+1}^k \in \C^{k-j}_{\A}\\ a\in S_{j+1}} }\prod\limits_{i=j+1}^k(x_i-1)^{\rk (M)-\rk (S_i)} (y_i-1)^{|S_i|-\rk(S_i)} \phi_j $$ where $$\phi_j=\sum\limits_{(S_i)_1^j \in \C^j_{S_{j+1}- a} }  \prod\limits_{i=1}^j(x_i-1)^{\rk (M)-\rk (S_i)} (y_i-1)^{|S_i|-\rk(S_i)}.$$ Notice that $\rk(M)=\rk_{M\bs [\A-(S_{j+1}-a)]}(S_{j+1}-a) +\rk_{M/(S_{j+1}-a)}(\A-(S_{j+1}-a))$. Using this we see that $$\phi_j=\left[\prod\limits_{i=1}^j(x_i-1)^{\rk (M/(S_{j+1}-a))}\right]T^j_{M\bs [\A-(S_{j+1}-a)]}.$$ Then we note that for $a\in S_{j+1}\subseteq \cdots \subseteq S_k$ we know that $\rk(M)-\rk(S_i)=\rk(M/a)-\rk_{M/a}(S_i-a)$ and that $|S_i|-\rk_M(S_i)=|S_i-a|-\rk_{M/a}(S_i-a)$ since $a$ is not a loop. Hence $$st^{k,j}_{M,a}=\sum\limits_{\substack{(S_i)_{j+1}^k \in \C^{k-j}_{\A}\\ a\in S_{j+1}} }\prod\limits_{i=j+1}^k(x_i-1)^{\rk(M/a)-\rk_{M/a}(S_i-a)} (y_i-1)^{|S_i-a|-\rk_{M/a}(S_i-a)} \phi_j .$$ Now that each term has $a$ taken out we reindex and write this as $$st^{k,j}_{M,a}=\sum\limits_{(S_i)_{1}^{k-j} \in \C^{k-j}_{\A -a} }\prod\limits_{i=1}^{k-j}(x_{i+j}-1)^{\rk(M/a)-\rk_{M/a}(S_i)} (y_{i+j}-1)^{|S_i|-\rk_{M/a}(S_i)} \phi_j' $$ where $$\phi_j'=\left[\prod\limits_{i=1}^j(x_i-1)^{\rk (M/S_{1})}\right]T^j_{M\bs [\A-S_{1}]}.$$ Since this is exactly the definition of our split Tutte polynomials we are done. Now the formulas for the cases where $a\in \A$ is a loop or coloop is much simpler. Also, the proofs are a direct application of Proposition \ref{Tprod}. \end{proof}

\section{Generalized permutahedra and Valuations}\label{val-sec}

We start with the definition of generalized permutahedra where we follow Ardila-Sanchez \cite{AS-20}. We will study a convolution product definition of a generalized a version of chain Tutte polynomials to the setting of generalized permutahedra. Then we will restrict this new definition to the submonoid of matroids in order to show how chain Tutte polynomials can be derived from these convolution products. 

\subsection{Generalized permutahedra} First we note that a function $f:2^\A\to \R$ is \emph{submodular} if for all $A,B\subseteq \A$ the function $f$ satisfies $f(A)+f(B)\geq f(A\cup B)+f(A\cap B)$. 

\begin{definition}

A \emph{generalized permutahedron} is a polytope $P$ in $\R^\A$ which is of the form $$P=\{x\in \R^\A\ |\ \sum\limits_{i\in \A} x_i=z(\A) \text{ and }  \sum\limits_{i\in S} x_i\leq z(S)\text{ for all } S\subseteq \A \}$$ where $z:2^\A \to \R$ is a submodular function. We denote the set species of generalized permutahedra by ${\bf GP}$ (meaning for any finite set $I$, ${\bf GP}[I]$ is a vector space spanned by the set of all generalized permutahedra in $\R^I$) which has coproduct $\Delta_{S_1,\dots ,S_k}$ defined by $\Delta_{S_1,\dots ,S_k}(P)=(P_1,\dots ,P_k)$ where $P_i$ are generalized permutahedra in $\R^{S_i}$ such that the maximal face of $P$ with respect to the indicator vector $e_{(S_1,\dots ,S_k)}=e_{S_1}+e_{S_1\sqcup S_2}+\cdots +e_{S_1\sqcup \cdots \sqcup S_k}$ is $P_{e_{(S_1,\dots ,S_k)}}=P_1\times \cdots \times P_k$ and (see \cite[Proposition 1.4.4]{AA-17} for details).

\end{definition}

Next we need to consider how polytopes decompose into pieces. The idea is to construct invariants on polytopes inductively by their pieces.

\begin{definition}\label{subdiv}

A generalized permutahedra \emph{subdivision} of a polytope $P$ is a set of generalized permutahedra $\{P_1,\dots ,P_s\}$ whose vertices are vertices of $P$, $$P=\bigcup_{i=1}^sP_i$$ and for all $1\leq i<j\leq s$ if $P_i\cap P_j\neq \emptyset $ then $P_i\cap P_j$ is a proper face of both $P_i$ and $P_j$.  

\end{definition}

Now we can state the definition of a valuation using subdivisions, which is the main subject of the section.

\begin{definition}\label{GP-val}

A function $f:{\bf GP} \to R$ where $R$ is an algebra is a \emph{valuation} if for any $P\in {\bf GP}$ and any subdivision $\{P_1,\dots ,P_k\}$ of $P$ we have $f(\emptyset )=0$ and $$f(P)=\sum\limits_{\emptyset \neq \{j_1,\dots ,j_i\}\subseteq [k]}\hspace{-.5cm}(-1)^{i}f(P_{j_1}\cap \cdots \cap P_{j_i}).$$

\end{definition}

\begin{remark}

In the literature there are many studies on different classes of combinatorial objects which properly contain all generalized permutahedra as well as many different notions of valuative functions (see \cite[Appendix A]{EHL-23} for an excellent summary). In this paper we restrict to ${\bf GP}$ and the definition of a valuation in Definition \ref{GP-val}. 

\end{remark}

Next we focus on some polynomials which are slightly different than those defined in Definition \ref{Def-Tutte}. To do this, again we follow Ardila and Sanchez \cite{AS-20}. In \cite[Definition 6.2]{AS-20} Ardila and Sanchez present a convolution of species maps $f_i:{\bf GP}\to R$ for $i\in [k]$ where $R$ is an algebra over $\F$ with multiplication $m$. The \emph{convolution} of these functions is the species map $f_1\star f_2 \star \cdots \star f_k: {\bf GP} \to R$ defined by $$f_1\star f_2 \star \cdots \star f_k [\A] (P)=\sum\limits_{S_1\sqcup \cdots \sqcup S_k =\A}m^{k-1}\circ f_1[S_1]\otimes \cdots \otimes f_k[S_k]\circ \Delta_{S_1,\dots ,S_k} (P)$$ where $f_i[S_i]:{\bf GP}[S_i] \to R$ are restrictions.

Now we present a chain Tutte polynomial defined on ${\bf GP}$ using a species map which can be found in the proof of Proposition 7.4 in \cite{AS-20}. 

\begin{definition}\label{GP-uniTutte-def}

Let $N_i[\A](P)=u_i^{|\A|}v_i^{z_P(\A)}$ where $z_P$ is the semimodular function defining $P$. The $k^{th}$ \emph{permutahedral chain Tutte polynomial} is 

$$\T^k_P((u_i);(v_i))=N_1\star N_2 \star \cdots \star N_k[\A] (P).$$

\end{definition}

The next fact follows from the construction in Definition \ref{GP-uniTutte-def} and results in \cite{AS-20}.

\begin{proposition}\label{uT-val-GP}

The $k^{th}$ permutahedral chain Tutte polynomial $\T^k$ is a valuation on ${\bf GP}$.

\end{proposition}

\begin{proof}

In \cite[Proposition 7.4]{AS-20} Ardila and Sanchez show that $N_i$ is a valuation on the species of extended generalized permutahedra. Then by \cite[Corollary 6.3]{AS-20} we have the conclusion.\end{proof}

\subsection{Matroid valuations}\label{val-subsec}

First we recall the basis matroid polytope (using \cite{AS-20} as our general reference for this material). A matroid $M$ can be defined via its set of bases $\B(M)$ which are all the independent sets of $M$ whose size is the rank of $M$. Then the matroid polytope of $M$ is $$P(M)=\mathrm{Conv}\{e_B|B\in \B(M)\}$$ where $e_B=e_{i_1}+\cdots +e_{i_r}$ with $B=\{i_1,\dots, i_r\}$. Now we need a few key definitions to state our main result. These are exactly the same as Definitions \ref{subdiv} and \ref{GP-val} restricted to the Hopf submonoid of matroids.

\begin{definition}

A \emph{matroid polyhedral subdivision} of a matroid polytope $P(M)$ is a subdivision of $P(M)$ where all the pieces are matroid polytopes.


\end{definition}

Now we want to know how invariants decompose across subdivisions which gives rise to valuations. We note again that there are many notions of valuations on matroids, but we call a function $f:\mathrm{Mat} \to R$ a \emph{matroid valuation} if it is a valuation of matroid polytopes.

%
%

Now we note that the monoid of matroids is a submonoid of ${\bf GP}$ so we can restrict our permutahedral chain Tutte polynomial to the monoid of matroids. In the case of matroids the coproduct on a set decomposition $\A=S_1\sqcup \cdots\sqcup S_k$ with $A_0=\emptyset$ and $A_i=S_1\sqcup \cdots \sqcup S_i$ is (\cite[Lemma 2.7]{AS-20}) $$\Delta_{S_1,\dots ,S_k}(M)=M[A_0,A_1] \otimes M[A_1,A_2] \otimes \cdots \otimes M[A_{k-1},A_k]$$ where $M[A,B]=(M|B)/A$. For matroid polytopes the associated semimodular function is exactly the rank function ($z_{P(M)}=\rk_M$). Then the $i^{th}$ term in the product of $\T^k$ is $$N_i[S_i](P(M[A_{i-1},A_i]))= u_i^{|S_i|}v_i^{\rk_{(M|{A_i})/A_{i-1}}(S_i)}.$$ Putting this all together the $k^{th}$ permutahedral chain Tutte polynomial of a matroid polytope $P(M)$ is \begin{align} \T^k_{P(M)}((u_i);(v_i))&=\sum\limits_{S_1\sqcup \cdots \sqcup S_k =\A}\left[\prod\limits_{i=1}^kN_i[S_i](P(M[A_{i-1},A_i]))\right]\\
&=\sum\limits_{(A_i)\in \C_{\A}^k}\prod_{i=1}^ku_i^{|A_i-A_{i-1}|}v_i^{\rk(A_i)-\rk(A_{i-1})} \label{uchianTutte-mat}\end{align}  where $A_0=\emptyset$ and $A_k=\A$ in the summation (\ref{uchianTutte-mat}).

\begin{proof}[Proof of Theorem \ref{T-universal}] First we make a change of coordinates by setting $u_{k+1}=1$, $v_1=b_1^{-1}\cdots b_k^{-1}$, $v_{k+1}=a_1\cdots a_k$, for $i\in [k]$ $$u_i=\prod_{j=i}^kb_j,$$ and for $1<i<k+1$  $$v_i=\prod_{j=i}^kb_j^{-1}\prod_{j=1}^{i-1}a_j.$$ Putting this change of coordinates into the formulation of $\T^{k+1}_{P(M)}$ in (\ref{uchianTutte-mat}) we get 

$$\T^{k+1}_{P(M)}((u_i);(v_i))=\sum\limits_{A_1,\dots ,A_{k+1}\in \C_{\A}^{k+1}}\prod_{i=1}^{k+1}u_i^{|A_i-A_{i-1}|}v_i^{\rk(A_i)-\rk(A_{i-1})}$$
$$=\hspace{-1cm}\sum\limits_{A_1,\dots ,A_{k+1}\in \C_{\A}^{k+1}}\prod_{i=1}^{k}\left(\prod_{j=i}^kb_j\right)^{|A_i-A_{i-1}|}\left( \prod_{j=i}^kb_j^{-1}\prod_{j=1}^{i-1}a_j\right)^{\rk(A_i)-\rk(A_{i-1})}\hspace{-1cm}(a_1\cdots a_k)^{\rk(A_{k+1})-\rk(A_{k})}.$$ In each of these terms we see that the exponent of $a_i$ is $$\sum\limits_{j=i}^{k}\rk(A_{j+1})-\rk(A_{j})=\rk(M)-\rk(A_{i})$$ since $A_{k+1}=\A$ and the exponent of $b_i$ is $$\sum\limits_{j=1}^{i}\big(|A_{j}-A_{j-1}| -(\rk(A_j)-\rk(A_{j-1}))\big)=|A_i|-\rk(A_{i}).$$ Hence with this change of coordinates $\T^{k+1}_{P(M)}((u_i)_1^{k+1};(v_i)_1^{k+1})$ is the $k^{th}$ chain Whitney polynomial $W^{k}_M((a_i)_1^{k};(b_i)_1^{k})$. Then restricting to the submonoid of matroids of ${\bf GP}$ using Proposition \ref{uT-val-GP} we can conclude $W^k_M$ and also $T^k_M$ are matroid valuations. \end{proof}

\section{Chain Tutte evaluations and specializations}\label{evals-sec}

In this section we study various evaluations and specializations of chain Tutte polynomials. One idea we visit repeatedly is to find an evaluation that gives some combinatorial information and then apply the recursion from Theorem \ref{recursion}. We start with some basic constant evaluations and then look at some polynomials. For most of this section we will just focus on $T^2$.

\subsection{Constant evaluations} We begin by discussing some constant evaluations of $T^k_M$ which are straightforward to describe combinatorially and are essentially obtained from the classical Tutte polynomial. The following can be concluded directly from the definition.

\begin{proposition}\label{nbasis}

For any matroid $M$ we have $$T^k_M(1,\dots,1;1,\dots , 1)=|\B |$$ where $\B$ is the set of bases of $M$.

\end{proposition}

%
%
%
Next we look at an example of a different constant evaluation for a specific matroid.

\begin{example}\label{Boolean-eval}

Let $\B_n$ be the Boolean matroid of rank $n$ (a.k.a.\ the uniform matroid of rank $n$ on $n$ elements). Then $T^k_{\B_1}(2,\dots , 2;1,\dots ,1)=k+1$ since a flag of subsets of a set with one element is determined by how many empty sets are in the sequence (also just evaluate the expression given in Example \ref{boolean-ex}). Using this and Theorem \ref{Tprod} we get that $$T^k_{\B_n}(2,\dots , 2;1,\dots ,1)=(k+1)^n.$$

\end{example}

The Boolean matroids are fairly simple to understand since there is only one basis. However, in the $k=2$ case we can reformulate the result of Example \ref{Boolean-eval} in general.

\begin{proposition}

Let $M$ be a matroid and $I_m$ be the number of independent sets of size $m$ in $M$. Then $$T^2_{M}(2, 2;1,1)=\sum\limits_{m=0}^{\rk(M)}2^mI_m.$$

\end{proposition} 

If we swap the 1s and 2s in the above we get the `dual' result. Recall that $T^1_{M}(1;2)$ is the number of spanning sets of the matroid. Now we examine $T^2_M(1,1;2,2)$ and again we need more notation. 

\begin{proposition}

Let $SP$ be the set of spanning sets of $M=(\A, \rk)$ and for $X\in SP$ let $NSP(X)$ be the number of subsets of $X$ that are also spanning. Then \begin{align*}T^2_M(1,1;2,2)&=\sum\limits_{X\in SP}NSP(X)\\
&=\sum\limits_{X\in SP}2^{|\A|-|X|} .\\
\end{align*}

\end{proposition}

The evaluation $T^2_M(1,2;2,1)$ is the number of bases because in the summation the only terms that contribute are those which $S_1$ must be spanning and $S_2$ must be independent. If $S_1\subseteq S_2$ with $S_1$ spanning and $S_2$ independent then $S_1=S_2$ and they are a basis. Now we examine the evaluation of swapping  1s and 2s in this last evaluation.

\begin{proposition}\label{2,1;1,2}

For $X\in SP$ a spanning set let $NI(X)$ be the number of independent sets inside $X$. Then $$T^2_M(2,1;1,2)=\sum\limits_{X\in SP}NI(X).$$

\end{proposition}

We can phrase Proposition \ref{2,1;1,2} in terms of graphs.

\begin{example}

Let $G$ be an undirected graph with $M_G$ the associated matroid. Then for any spanning subgraph $G'$ the number of independent sets inside this subgraph is the number of subforests. Hence $$T^2_{M_G}(2,1;1,2)=\sum\limits_{G' \text{ spanning}}NF(G')$$ where $NF(G')$ is the number of subforests in $G'$. These numbers seem interesting for complete graphs $K_n$ and cycle graphs $C_n$. Using the computer algebra system Sage \cite{sage} we compute these evaluations in Tables \ref{complete-2,1;1,2} and \ref{cycle-2,1;1,2}.

\begin{center}
\begin{table}
\begin{tabular}{|l|c|c|c|c|c|}
\hline
$n$&1&2&3&4&5\\
\hline
$T^2_{M_{K_n}}(2,1;1,2)$ &1&2&19&523&36478\\
\hline
\end{tabular}\caption{The (2,1;1,2) evaluations for complete graphs}\label{complete-2,1;1,2}
\end{table}
\end{center}

\begin{center}
\begin{table}
\begin{tabular}{|l|c|c|c|c|c|}
\hline
$n$&3&4&5&6&7\\
\hline
$T^2_{M_{C_n}}(2,1;1,2)$ &19&47&111&255&575\\
\hline
\end{tabular}\caption{The (2,1;1,2) evaluations for cycle graphs}\label{cycle-2,1;1,2}
\end{table}
\end{center}

\end{example}

Now we consider a few evaluations with $k>0$ instead of just $k=2$.

%
%
%

\begin{proposition}\label{1100prop}

If $M$ is a matroid with minimal flat $\hat{0}$ and maximal flat $\hat{1}$ then $$T^k_M(1,1,\dots, 1;0,0,\dots ,0)=\left\{\begin{array}{cl}1 & \text{ if } k\text{ is even}\\ (-1)^{\rk (M)}\gm (\hat{0},\hat{1}) & \text{ if } k\text{ is odd} \end{array}\right. .$$

\end{proposition}

\begin{proof}

We compute $$T^k_M(1,1,\dots, 1;0,0,\dots ,0)=(-1)^{k\cdot \rk(M)}\hspace{-.5cm}\sum\limits_{\substack{S_1\subseteq \cdots \subseteq S_k\subseteq \A \\ \bigvee S_1=\cdots =\bigvee S_k= \hat{1}}}(-1)^{\sum |S_i|} .$$ Next we induct on $k$. The classic $k=1$ case can be found for example in \cite[(6.21)]{BO92} or can be deduced from Lemma \ref{m1}. The $k=2$ base case is just Lemma \ref{mobius-Whitneythm} for $\gm (\hat{1},\hat{1})$. When $k$ is even we have a telescoping decomposition $$\sum\limits_{\substack{S_1 \subseteq \cdots \subseteq S_k\subseteq \A\\ \bigvee S_1=\cdots =\bigvee S_k= \hat{1}}}(-1)^{\sum |S_i|}=\sum\limits_{\substack{S_{k-1}\subseteq S_k\subseteq \A \\ \bigvee S_{k-1}=\bigvee S_k =\hat{1}}}(-1)^{|S_{k-1}|+|S_k|}\left[\sum\limits_{\substack{S_1 \subseteq \cdots \subseteq S_{k-1}\\ \bigvee S_1=\cdots =\bigvee S_{k-1}}}(-1)^{\sum\limits_{i=1}^{k-2} |S_i|}\right].$$ Then by induction the summation in the above bracket is 1. So, using Lemma \ref{mobius-Whitneythm} again we get the total is 1.
Now suppose that $k>2$ is odd. Similarly we separate the sum $$\sum\limits_{\substack{S_1 \subseteq \cdots \subseteq S_k\subseteq \A\\ \bigvee S_1=\cdots =\bigvee S_k= \hat{1}}}(-1)^{\sum |S_i|}=\sum\limits_{\substack{ S_k\subseteq \A \\ \bigvee S_k =\hat{1}}}(-1)^{|S_k|}\left[\sum\limits_{\substack{S_1 \subseteq \cdots \subseteq S_{k-1}\subseteq S_{k}\\ \bigvee S_1=\cdots =\bigvee S_{k}}}(-1)^{\sum\limits_{i=1}^{k-1} |S_i|}\right].$$ Since $k-1$ is even the summation in the bracket above is 1 by induction and then the result follows from Lemma \ref{m1}. \end{proof}

Now switching the zeros and ones we examine the `dual' result whose proof is very similar. Recall that $T^1_M(0;1)$ is the reduced Euler characteristic of the independence complex of $M$ (see \cite[Section 7.7.1]{Ardila-15}).

\begin{proposition}

If $M$ is a matroid with $\chi$ the reduced Euler characteristic of the independence complex of $M$ and $k\geq 1$ then $$T^k_M(0,0,\dots, 0;1,1,\dots ,1)=\left\{\begin{array}{cl}1 & \text{ if } k\text{ is even}\\ (-1)^{\rk(M)}\chi & \text{ if } k\text{ is odd} \end{array}\right. .$$

\end{proposition}

\begin{proof}

Assume that $k>1$ is even. In this case there is again a telescoping of the sum but this time the sets are all independent:

$$T^k_M(0,0,\dots, 0;1,1,\dots ,1)=(-1)^{k\cdot \rk(M)}\hspace{-.5cm}\sum\limits_{\substack{S_1\subseteq \cdots \subseteq S_k\subseteq \A \\ \text{independent}}}(-1)^{\sum |S_i|}$$

$$=(-1)^{k\cdot \rk(M)}\hspace{-.5cm}\sum\limits_{\substack{S_k\subseteq \A \\ \text{independent}}}(-1)^{|S_k|}\sum\limits_{\substack{S_{k-1}\subseteq S_k \\ \text{independent}}}(-1)^{|S_{k-1}|}\cdots \sum\limits_{\substack{S_1\subseteq S_2 \\ \text{independent}}}(-1)^{|S_1|}.$$  The individual sum $$\sum\limits_{\substack{S_1\subseteq S_2 \\ \text{independent}}}(-1)^{|S_1|}$$ is 0 unless $S_2=\emptyset$. If $S_2=\emptyset$ then this pair  $$\sum\limits_{\substack{S_2\subseteq S_3 \\ \text{independent}}}(-1)^{|S_2|}\left[\sum\limits_{\substack{S_1\subseteq S_2 \\ \text{independent}}}(-1)^{|S_1|}\right]=(-1)^{|\emptyset|}=1.$$ Since $k$ is even each pair reduces to 1 and we have the desired result. Finally if $k$ is odd then the same decomposition has all the pairs inside the sum reduce to 1 and the last summation term $$\sum\limits_{\substack{S_k\subseteq \A \\ \text{independent}}}(-1)^{|S_k|}$$ remains as is exactly the reduced Euler characteristic (again see \cite[Section 7.7.1]{Ardila-15}). \end{proof}


\subsection{The M\"obius polynomial} The M\"obius polynomial was used by Jurrius in \cite{Jur-12} to study weight enumerators of error-correcting linear codes. Then recently Johnsen and Verdure studied the M\"obius polynomial on error-correcting codes from a commutative algebra view point in \cite{JV-21}. In this subsection we study the M\"obius polynomial exclusively.

\begin{proof}[Proof of Theorem \ref{evalT2-mob}]

First we compute the evaluation of the second chain Tutte polynomial: 

\begin{align*}
T^2_{L(M)}(1-s,1-t;0,0)=& \sum\limits_{A\subseteq B\subseteq E } (-s)^{\crk(A)}(-t)^{\crk(B)}(-1)^{|A|+|B|-\rk(A)-\rk(B)}\\
=&  \sum\limits_{A\subseteq B\subseteq E }(-1)^{|A|+|B|} s^{\crk(A)} t^{\crk(B)}\\
=& \sum\limits_{X\leq Y\in L(M)}\left[  \sum\limits_{\substack{A\subseteq B\subseteq E \\ \bigvee A =X \\ \bigvee B=Y }}(-1)^{|A|+|B|}\right] s^{\crk(X)}t^{\crk(Y)}.
\end{align*}
Then by Lemma \ref{mobius-Whitneythm} we are done since that is exactly the definition of the M\"obius polynomial (see (\ref{Mobdef})). We also get to conclude that the M\"obius polynomial is a valuation from combining Theorem \ref{chainTutte-val} and Proposition \ref{evalT2-mob}. \end{proof}

We see that the M\"obius polynomial is not a Tutte-Grothendieck invariant by using the evaluation in Proposition \ref{evalT2-mob} applied to the sequence of matroids in Example \ref{not-Tutte-inv}. However, applying the recursion from Theorem \ref{recursion} and the evaluation for $T^2$ from Theorem \ref{evalT2-mob} we get a new recursion for the M\"obius polynomial. 

\begin{corollary}\label{Mob-recur}

If $M=(\A,\I)$ is a matroid and $a\in \A$ is not a loop or coloop then \begin{align*}\bar{\chi}_M(s,t)=&\bar{\chi}_{M\bs a}(s,t) +\bar{\chi}_{M/a}(s,t)\\
&+ (-1)^{\rk(M/a)}\sum\limits_{S\subseteq \A-a} (-1)^{\rk(M/S)+|S|}s^{\rk(M/S)}t^{\rk(M/a)-\rk(S)}\chi_{M\bs [\A-S]}(s).\end{align*}

\end{corollary}

Putting Proposition \ref{evalT2-mob} and Proposition \ref{1100prop} together gives a quick proof that $\bar{\chi}_M(0,0)=1$ (which is an easy consequence of $\gm (\hat{1},\hat{1})=1$).


\subsection{Opposite characteristic polynomial}

Let $M$ be a matroid and $L(M)$ its lattice of flats. The order on $L(M)$ is given by inclusion. If we just reverse this inclusion order we get the lattice $L(M)^{\op}$ which is the same set of flats but with reverse inclusion order (in the case of a hyperplane arrangement $L^{\op}$ would be the regular inclusion order on intersections). This opposite order on flats is not necessarily atomic nor semimodular. However $L(M)^{\op}$ is still a ranked lattice and so we can consider its characteristic polynomial.

\begin{definition}

The \emph{opposite characteristic polynomial} of a matroid $M$ with top flat $\hat{1}$ is \[\chi^{\op}_M(t)=\chi_{L(M)^{\op}}(t)=\sum\limits_{X\in L(M)^{\op}}\gm (\hat{1},X)t^{\crk^{\op}(X)}\] where $\crk^{\op}(X)=\rk(X)$ (because of the opposite order).

\end{definition}

Even this polynomial has an evaluation formula using the second chain Tutte polynomial.

\begin{proposition}\label{opchar-eval}

If $M$ is a matroid then $$T^2_M(1-t,1;0,0)=t^{\rk(M)}\chi^{\op}_{L(M)}(t^{-1}).$$

\end{proposition}

\begin{proof}

We compute the evaluation \begin{equation}\label{opchar1}T^2_M(1-t,1;0,0)=\sum\limits_{\substack{A\subseteq B \subseteq \A \\ B \text{ spanning}}} (-1)^{\rk(M)-\rk(A)+|A|+|B|-\rk(A)-\rk(B)}t^{\rk(M)-\rk(A)}.\end{equation} Since $B$ is spanning we know that $\rk(B)=\rk(M)$ hence (\ref{opchar1}) becomes \begin{equation}\label{opchar2}\sum\limits_{\substack{A\subseteq B \subseteq \A \\ B \text{ spanning}}} (-1)^{|A|+|B|}t^{\rk(M)-\rk(A)}.\end{equation} Then we split the sum of (\ref{opchar2}) on $A$ by flats as \begin{equation}\label{opchar3} \sum\limits_{X\in L(M)} \left(\sum\limits_{\substack{A\subseteq B \subseteq \A \\ \bigvee A=X \\ \bigvee B=\hat{1}}} (-1)^{|A|+|B|}\right)t^{\rk(M)-\rk(X)}.\end{equation} Applying Lemma \ref{mobius-Whitneythm} to the inner sum of (\ref{opchar3}) we get \begin{equation}\label{opchar4} \sum\limits_{X\in L(M)} \gm (X,\hat{1})t^{\rk(M)-\rk(X)}.\end{equation} Since the M\"obius function can be defined recursively from the top element $\hat{1}$ we have that (\ref{opchar4}) becomes \begin{equation}\label{opchar5} t^{\rk(M)}\sum\limits_{X\in L(M)^{\op}} \gm (X,\hat{1})(t^{-1})^{\rk(X)}.\end{equation} Since the corank in $L(M)^{\op}$ is actually the rank in $L(M)$ the expression in (\ref{opchar5}) is actually $t^{\rk(M)}\chi^{\op}_{L(M)}(t^{-1})$ and we have proved the result. \end{proof}

Again we get to combine the evaluation result, in this case Proposition \ref{opchar-eval}, with Theorem \ref{chainTutte-val} to conclude the opposite characteristic polynomial is a valuation.

\begin{corollary}\label{opchar-val}

The opposite characteristic polynomial is a valuation on matroids.

\end{corollary}

We also restrict the recursion from the chain Tutte polynomial in Theorem \ref{recursion} to the opposite characteristic polynomial. Again note that the opposite characteristic polynomial is not a Tutte-Grothendieck invariant (use the evaluation in Proposition \ref{opchar-eval} with Example \ref{not-Tutte-inv}).

\begin{corollary}\label{charop-recur}

If $M=(\A,\rk)$ is a matroid, $a\in\A$ is not a loop or coloop, and $\mathrm{SPD}=\{S\subseteq \A -a |\ S\ \text{spans}\ M/a\}$ then $$\chi^{\op}_{M}(t)=\chi^{\op}_{M\bs a}(t)+\chi^{\op}_{M/a}(t)+(-1)^{\rk(M)-1}\hspace{-.2cm}\sum\limits_{S\in \mathrm{SPD}}(-1)^{|S|}\chi_{M|S}(t).$$

\end{corollary}

\begin{proof} Because of the formula in Proposition \ref{opchar-eval} we plug that evaluation into Theorem \ref{recursion}. We simplify the split Tutte polynomial \begin{align*}sT^{2,1}_{M,a}(1-t,1;0,0)&=\sum\limits_{S\subseteq \A-a}T^1_{M|S}(1-t;0)(0)^{\rk(M/a)-\rk_{M/a}(S)}(-1)^{|S|-\rk_{M/a}(S)}(-t)^{\rk(M/S)}\\
&=\sum\limits_{S\in \mathrm{SPD}}\chi_{M|S}(t)(-1)^{|S|-\rk(M/a)}(-t)^0\\ 
\end{align*}which finishes the proof. \end{proof}

From this recursive result we can deduce the next corollary (even though it can be derived from the definition and properties of the M\"obius function).

\begin{corollary}

If $M$ is a simple matroid with $\rk(M)>0$ then $\chi^{\op}_M(1)=0$.

\end{corollary}

\begin{proof}

We proceed by induction on the number of ground set elements of $M$. If there is one element and the rank is one then $\chi^{\op}_M(t)=t-1$. Then for larger ground sets the recursion in Corollary \ref{charop-recur} finishes the proof.\end{proof}


\subsection{Generalized M\"obius polynomial}

In \cite{JW-20} a function $J$ was studied that generalizes the M\"obius function and it was used to define and study a generalized M\"obius polynomial. Also in \cite{JW-20} it was conjectured that this generalized M\"obius polynomial was a matroid valuation. We prove that conjecture in this subsection.

\begin{definition}

The $J$-function on a matroid $M$ with ground set $\A$ is $J:\Fl_\A^3\to R$ where $R$ is a commutative ring and $\Fl_\A^3=\{(X,Y,Z)\in L(M)^3| X\leq Y\leq Z \}$ is 3-flags of flats in $L(M)$ defined recursively by $J(X,X,X)=1$ for all $X\in L(M)$ and $$\sum\limits_{X\leq A\leq Y\leq B\leq Z}J(A,Y,B)=\gd_3(X,Y,Z)$$ for all $(X,Y,Z)\in \Fl_\A^3$ where $\gd_3$ is the 3-variable Kronecker delta function.

\end{definition}

In \cite{JW-20} it was shown that this $J$-function satisfies various generalizations of the M\"obius function. This led the authors there to define analogous generalized characteristic and M\"obius polynomials.

\begin{definition}

The generalized $J$-M\"obius polynomial of $M$ is $$\M_M(t)=\sum\limits_{(X,Y,Z)\in \Fl^3(L(M))}J(X,Y,Z)t^{\crk (X)+\crk(Y)+\crk(Z)}.$$

\end{definition}

Also in \cite[Proposition 6.10]{JW-20} it was shown that this generalized $J$-M\"obius polynomial satisfies a decomposition with the classic characteristic polynomials.

\begin{proposition}[\cite{JW-20}]\label{JW-prop-GMob}

If $M$ is a matroid then $$\M_M(t)=t^{\rk (M)}\sum\limits_{y\in L(M)}t^{\crk(y)}\chi_{L(M)^y}(t)\chi_{(L(M)^\op)^y}(t^{-1}).$$

\end{proposition}

Now putting together Proposition \ref{JW-prop-GMob} and Corollary \ref{opchar-val} we can answer \cite[Question 6.26]{JW-20}.

\begin{corollary}

The generalized $J$-M\"obius polynomial is a valuation on matroids.

\end{corollary}


\subsection{Expected codimension and Ford's $S$-polynomial}\label{G-inv-sub}

Understanding properties in general of the realization space of a matroid (a.k.a. matroid variety) is an elusive endeavor (see \cite{Va-06} and \cite{STW-21}). In \cite{Ford-15} Ford studied the dimension of matroid varieties and when the dimensions could be computed combinatorially. To do this Ford defined the following.

\begin{definition}[\cite{Ford-15}]

Let $M=(\A,\rk)$ be a matroid and $a:2^{\A}\to \ZZ$ be defined recursively by $a(\emptyset)=0$ and $$a(A)=\crk(A)-\sum\limits_{B\subset A}a(B).$$ The \emph{expected codimension of} $M$ is $$ec(M)=\sum\limits_{A\subseteq \A}(\rk(M)-\rk(A))a(A).$$

\end{definition}

The idea behind this definition is that counting Pl\"ucker coordinates gives a way to guess the codimension of the matroid variety. Then Ford in \cite{Ford-15} proves that for positroids this expected codimension is actually the codimension of the positroid's realization space. The final result in Ford's paper on this expected dimension is that it is a valuative property. Hence a natural question is how we can produce this expected dimension from these chain Tutte polynomials. It turns out we can do this via another polynomial which Ford defines in \cite[Lemma 5.4]{Ford-15}.

\begin{definition}[\cite{Ford-15}]

Let $M=(\A,\rk)$ be a matroid. The \emph{Ford $S$-polynomial of $M$} is $$S_M(x,y,z)=\sum\limits_{A\subseteq B\subseteq \A}x^{|A|-\rk(A)}y^{\rk(M)-\rk(B)}z^{|B|-|A|}.$$

\end{definition}

Next in \cite{Ford-15} Ford proves that this polynomial $S_M(x,y,z)$ is a matroid valuation using a similar technique to Speyer's proof in \cite{Speyer08} that the Tutte polynomial is a matroid valuation. It turns out that we can find this $S_M$ polynomial in $T^2$ which gives another proof that it is a matroid valuation.

\begin{proposition}[{\cite[Lemma 5.4]{Ford-15}}]\label{S_M-val}

The Ford $S$-polynomial is a matroid valuation.

\end{proposition}

\begin{proof} Since $T^2$ is a matroid valuation and $$T^2_M(z+1,yz^{-1}+1;xz^{-1}+1,z+1)=S_M(x,y,z)$$ we have concluded the result.\end{proof}

Using the proof of \cite[Theorem 5.3]{Ford-15} we get the formula for expected codimension in terms of $T^2$.

\begin{corollary}\label{ec-T^2}

For any matroid $M$ $$ec(M)=\frac{\partial}{\partial x}\frac{\partial}{\partial y}\left[T^2_M(z+1,yz^{-1}+1;xz^{-1}+1,z+1)\right] (1,1,-1).$$

\end{corollary}

Now examine the recursive result on $T^2$, Theorem \ref{recursion}, applied to Ford's $S$-polynomial.

\begin{proposition}\label{S-del-con}

If $M=(\A,\rk)$ is a matroid and $a\in \A$ is not a loop or coloop then \begin{align*}S_M(x,y,z)=&S_{M\bs a}(x,y,z)+S_{M/a}(x,y,z)\\
&+z\sum\limits_{A\subseteq B\subseteq \A-a}  x^{|A|-\rk(A)}   y^{\rk (M/a)-\rk_{M/a}(B)}z^{|B|-|A|} \ .\end{align*}

\end{proposition}

Note that the recursion of Proposition \ref{S-del-con} appears very close to $S_{M\bs a}(x,y,z)+(1+z)S_{M/ a}(x,y,z)$. However this is not correct and if it were correct then it would imply that $ec(M)=ec(M\bs a)$ which is also not true in general.

Now putting together Proposition \ref{S-del-con} and Corollary \ref{ec-T^2} we have a new recursion for the expected codimension of a matroid.

\begin{corollary}

If $M=(\A,\rk)$ is a matroid with $a\in \A$ not a loop or coloop then $$ec(M)=ec(M\bs a)+ec(M/ a)-\sum\limits_{A\subseteq B\subseteq \A}(|A|-\rk (A))(\rk (M/ a) -\rk_{M/ a}(B))(-1)^{|B|-|A|}.$$

\end{corollary}

\bibliographystyle{amsplain}

\bibliography{whitneyb}

\end{document}